\documentclass[11pt]{amsart}
\usepackage{amsmath,amssymb,amscd}
\usepackage{epsfig}
\usepackage{epsfig, graphicx, pinlabel}

\usepackage{amsmath,amssymb,amscd}
\newtheorem{theorem}{Theorem}

\newtheorem{prop}{Proposition}
\newtheorem{corollary}{Corollary}
\newtheorem{lemma}{Lemma}

\newtheorem*{thma}{Theorem}

\theoremstyle{definition}

\newtheorem{eg}{Example}

\newtheorem{defn}{Definition}

\def\a{\alpha}
\def\b{\beta}

\def\bR{\mathbb R}

\def\rk{\textup{rank}\, }

\input xy
\xyoption{all}

\input epsf
\def\bx{{\bf x}}
\def\be{{\bf e}}
\def\row{\mathrm{row}}

\begin{document}

\title[characteristic polynomial of hyperplane arrangements]{Characteristic polynomial of certain hyperplane arrangements through graph theory}
\date{\today}
\author{Joungmin Song}
\address{Division of Liberal Arts \& Sciences\\GIST\\ Gwangju, 500-712, Korea}

%\subjclass{primary 32S22; secondary 05C30}
%\classification{primary 32S22; secondary 05C30}
%% Four or five keywords or phrases
\thanks{This work is supported by GIST Research Fund}
\email{songj@gist.ac.kr}
\begin{abstract}
We give a formula for computing the characteristic polynomial for certain hyperplane arrangements in terms of the number of bipartite graphs of given  rank and cardinality.

\end{abstract}
\maketitle

%\support{This research is supported by GIST FARE Research Fund}

\section{Introduction}

In this paper, we continue our study of the hyperplane arrangements $\mathcal J_n$ introduced in
\cite{SONG} which consists of the {\it type I walls}
\[
H_{\a\b} := \{ \bx \in \bR^n \, | \, x_\a + x_\b = 1\} = H_{\b\a}, \quad 1 \le a, b \le n
\]
and the {\it type II walls}
\[
0_i := \{\bx \in \bR^n  \, | \, x_i = 0\}, \mbox{ and } 1_i:=\{\bx \in \bR^n  \, | \, x_i = 1\}, 1 \le i \le n.
\]
This has a strong resemblance to the Shi arrangements \cite{Shi}.  The eventual goal of the project is to give a complete formula for the number of chambers i.e. the connected components of the complement of the hyperplanes. In the previous work,  we associated a colored graph to each hyperplane sub-arrangement of $\mathcal J_n$, and then gave an interpretation of the centrality condition of the arrangements in terms of graph properties \cite[Theorem~1]{SONG}.

Due to \cite{Zas75}, to compute the number of chambers, one has to enumerate the  central sub-arrangements of a given rank: the characteristic polynomial of $\mathcal J_n$ is defined
\[
\chi_{\mathcal J_n}(t) = \sum_{\mathcal B \subset \mathcal J_n \, \, \mbox{central}} (-1)^{\left|\mathcal B\right|} t^{n-\rk(\mathcal B)}
\]
and the number of chambers (resp. relatively bounded chambers) equals $(-1)^n \chi_{\mathcal J_n}(-1)$ (resp. $(-1)^{\rk \mathcal J_n}\chi_{\mathcal J_n}(+1)$). Let $\gamma_{k,s}$ be the number of central sub-arrangements of rank $k$ and cardinality $s$. Then  we have

\begin{equation}\label{E:char}
\chi_{\mathcal J_n}(t) = \sum_{k,s} (-1)^s \gamma_{k,s} t^{n-k}. \tag{$\dagger$}
\end{equation}

Due to \cite[Theorem~1]{SONG}, we may equivalently enumerate the central graphs of given rank and cardinality.
The rank of a $3$-colored graph equals, by definition, the rank of its {\it c-incidence matrix} (Definition~\ref{D:incidence}). As the first step, we give the following general formula for the rank of $3$-colored graphs:

\begin{thma} Let $G$ be a $3$-colored graph on $n$ vertices, and let $G'$ be the maximal subgraph none of whose vertices are connected to a colored vertex. Then the rank of $G$ equals
\[
n - (\mbox{number of bipartite components of} \, \,  G').
\]
\end{thma}
This will be proved in Sections~\ref{S:rank} and \ref{S:char} (Theorem~\ref{T:main}). One nice byproduct is a proof of that the rank of a color-less graph is full if and only if it has an odd cycle. Most standard textbooks in graph theory treat the rank of incidence matrix for the signed case only\footnote{In fact, all textbooks that I examined treat the signed case only.}, where the matrix is considered modulo $2$ and the analysis is considerably simpler. Although this certainly has been known to experts (see \cite{vanNuffelen}, e.g.), our method has its merit in that it clearly explains how odd cycles force the rank to be full and how even cycles may be cut without altering the rank. Hence a bipartite graph and its spanning trees have the same rank (Proposition~\ref{P:rank-spanning-tree}).

We also fulfill our promise we made in \cite{SONG} of utilizing more graph theory: in Section~\ref{S:formula}, we give a formula for computing the number of central hyperplane subarrangements of $\mathcal J_n$ in terms of the number of bipartite graphs of given rank and cardinality (Definition~\ref{D:card}). Counting the number of bipartite and various types of graphs has been extensively researched by many authors \cite{Gainer-Dewar, Hanlon, Harary58, Harary73, Ueno}. We shall utilize some of these results in a forthcoming work to give a more comprehensive formula.

The strategy for obtaining our formula is to decompose $3$-colored graphs into colored and colorless parts, and then find a relation between the number of the colored central graphs and that of the color-less  central graphs. We make this more precise in the rest of the introduction. As in the proof of \cite[Theorem~1]{SONG}, we decompose a $3$-colored graph $G$  into  three subgraphs $G', G'', G'''$:
\begin{enumerate}
\item ({\it graph of the first kind}) $G'$ is the union of colorless connected components;
\item ({\it graph of the second kind}) $G''$ is the union of isolated colored vertices;
\item ({\it graph of the third kind}) $G''' = G\setminus(G'\cup G'')$ is the union of the connected components with at least one colored vertex and at least one edge.
\end{enumerate}
The most intersting part is the enumeration of the central graphs of the third kind, and it is really the key result of this article. Let
$E'_{m,k,s}$ be the set of the connected, bipartite (colorless) graphs  on $[m]$ whose rank is $k$ and cardinality is $s \ge 1$, and let $E'''_{m,k,s}$ be the set of the connected, central graphs of the third kind on $[m]$ whose rank is $k$ and {\it cardinality} (Definition~\ref{D:card}) is $s$.

\begin{thma}\label{T:main} We have
\[
 |E'''_{k,k,s}| =  \sum_{t=1}^{s-k+1} 2 |E'_{k,k-1,s-t}| \binom kt.
 \]
\end{thma}

From this, we can readily obtain the following formula for counting the number of central graphs of given rank and cardinality in terms of the number of bipartite graphs. Let $\nu'_{m,n}$ (resp. $\nu^b_{m,n}$) be the number of connected graphs (resp. connected bipartite graphs) of order $m$ and size $n$.
\begin{thma}
The number of central graphs on $[n]$ of rank $k$ and cardinality $s$ equals
\[
 \begin{array}{cll}
 \gamma_{k,s} & = &
\sum_{n, s}\binom n{n_b,n_{nb},n_2,n_3} 2^{n_2} \cdot \left( \sum_{n_b, s_b} \binom{n_b}{n_{b1},n_{b2},\dots,n_{b\ell}}' \prod\nu^b_{n_{bi}, s_{bi}}\right) \\
 & & \cdot \left( \sum_{n_{nb}, s_{nb}} \binom{n_{nb}}{n_{nb1},n_{nb2},\dots,n_{nb\ell'}}' \prod(\nu'_{n_{nbi}, s_{nbi}} - \nu^b_{n_{nbi},s_{nbi}})\right) \\
& & \cdot \left(\sum_{n_3, s_3} \binom{n_3}{n_{31},n_{32},\dots,n_{3\ell_3}}'\prod \left(\sum_{t_i=1}^{s_{3i}-n_{3i}} 2 \nu^b_{n_{3i}-1, s_{3i}-t_i} \binom{n_{3i}}{t_i} \right)\right)
\end{array}
\]
where the sum runs over all partitions
\[
\begin{array}{c}
n = n_b+n_{nb}+n_2+n_3; s = s_b+ s_{nb}+ s_2+s_3 \\
n_b=\sum_{i=1}^\ell n_{bi}; s_b = \sum_{j=1}^\ell s_{bj} \\
n_{nb}=\sum_{i=1}^{\ell'} n_{nbi}; s_{nb} = \sum_{j=1}^{\ell'} s_{nbj} \\
n_3=\sum_{i=1}^{\ell_3} n_{3i}; s_3=\sum_{j=1}^{\ell_3} s_{3j} \\
\end{array}
\]
such that the rank condition $k = n - \ell$ is satisfied (see Theorem~\ref{T:rank}).
\end{thma}
The above theorems are proved in Section~\ref{S:formula} (Theorem~\ref{T:formula} and Theorem~\ref{T:main}). The reduced multinomial $\binom a{a_1, a_2, \dots, a_r}'$ is defined in Definition~\ref{D:red-multinomial}. This formula provides a mean to compute the characteristic polynomial purely in terms of  the (bipartite) colorless graphs of given rank and size. In a forthcoming paper, we shall prove a characteristic polynomial formula using the graph enumeration results mentioned above as well as implementing the formula with a computer algebra system.

\section{The rank of the c-incidence matrices}\label{S:rank}

We let $\be_j$ denote the row vector with $1$ at the $j$th place and zero elsewhere. By a $3$-colored graph, we shall mean a graph $G = (V, E)$  with a color function $\gamma : V \to \{\pm1, 0\}$, with the convention that a vertex $v$ with $\gamma(v) = 0$ is {\it colorless}\footnote{In \cite{SONG}, the color function takes values in $\{0, 1, *\}$ with $*$ denoting the non-color, but for the purpose of computing  ranks, $\{\pm1, 0\}$ is judicious.}.
%Recall the incidence matrix:
%\begin{defn} Let $G = (V,E)$ be a graph with ordered vertices (assumed to be $1, 2, \dots, n$ without loss of generality) and edges.  The \emph{incidence matrix} $J_G$ of $G$ is the $|V|\times |E|$ matrix whose $k$th column is $\be_i + \be_j$ if the $k$th edge is $\{i,j\}$.
%\end{defn}

For a colored graph $G=(V,E)$, let $G_{mono}$ denote the colorless graph obtained by forgetting the colors on $V$.

\begin{defn}\label{D:incidence} Let $G=(V,E)$ be a colored graph on $n$ ordered vertices $\{1, 2, \dots, n\}$ and ordered edges. Let $\gamma$ be the color function on $V$. The {\it c-incidence} matrix $J_G$ is the $(|E|+n)\times n$-matrix whose top $|E|\times n$ submatrix is the transpose of the incidence matrix of $G_{mono}$ and whose bottom $n\times n$ submatrix has $i$th row equal to $\gamma(i)\be_i$. We define the \emph{rank} of the graph $G$ to be the rank of its c-incidence matrix.% with all $G(i)$ replaced by $1$.
\end{defn}
%When we speak of the rank, the colors $G(i)$ are assumed to be $1$.
Since changing the ordering of vertices in $V$ (resp. edges in $E$) is equivalent to multiplying the corresponding permutation matrix on the right (resp. left) of the c-incidence matrix, the rank of $J_G$ is independent of the choice of orderings on $V$ and $E$.
Note that $J_G$ equals (the transpose of) the regular incidence matrix with $n$ zero rows attached in the bottom if $G$ is not colored. In particular, the rank of $J_G$ is equal to the rank of the incidence matrix.

We begin by recording some basic, immediate observations about incidence matrix of colorless graphs.

\begin{eg}\label{E:basic}
\begin{enumerate}
\item It is immediately seen that  a linear graph $T$ on $n$ vertices has $\rk(J_T) = n-1$ since $J_T$ is an upper-triangular matrix.
\item Let $G$ be an odd cycle $(1, 2, 3, \dots, 2k-1, 1)$ on $n = 2k-1$ vertices. Order the edges naturally: $\{1,2\}$, $\{2,3\}$, $\dots$, $\{n-1,n\}$ followed by $\{n,1\}$, and consider the incidence matrix. Adding $\sum_{i=1}^{n-1}(-1)^i \row_i$ to $\row_{n}$ makes $\row_n$ equal to $2\be_n$. Then it is plain that the resulting matrix has rank $n$.
\item Let $G$ be an even cycle $(1, 2, 3, \dots, 2k, 1)$ on $n = 2k$ vertices. Order $E$ naturally as before then $\sum_{i=1}^{2k} (-1)^i\row_i = 0$.  Hence $\rk(J_G) = n-1$.
\end{enumerate}
\end{eg}

\begin{lemma}\label{L:tree} If $G$ is a tree with $n$ vertices, $\rk J_G = n-1$.
\end{lemma}

\begin{proof} We use induction on $n$, with the assertion trivially holding for the $n=2$ case. Any tree has at least two leaves. Rename vertices $1, \dots, n$ so that $1$ is a leaf and $\{1,2\}$ is the first edge. Then the incidence matrix is of the form
\[
\left[
\begin{tabular}{c|c c c c}
              % after \\: \hline or \cline{col1-col2} \cline{col3-col4} ...
              1 & 1 & 0 & $\cdots$ & 0 \\ \hline
              0 &  &  &  &  \\
              0 & &   &$J_{G'}$   &   \\
0 &        &   &  & \\
0 &        &    &       &
\end{tabular}
\right]
\]
where $G'$ is the tree obtained by deleting $1$ and $\{1,2\}$. By induction hypothesis, $\rk(J_{G'}) = n-2$. Plainly, $\rk(J_G) = \rk(J_{G'}) +1$.
\end{proof}

Note that a connected $3$-colored graph has a colored vertex if and only if its c-incidence matrix has a row that is a nonzero constant multiple of $\be_j$ for some $j$. Since the graph is connected, one can start with $\be_j$ and run the Gaussian elimination and obtain the reduced row echelon form which should be of full rank. Hence:

\begin{lemma} Let $G$ be a connected, $3$-colored graph. If it has a colored vertex, then its incidence matrix is of full rank.
\end{lemma}

\begin{prop} If a connected ($3$-colored) graph $G$ has an odd cycle, $\rk(J_G) = n$.
\end{prop}
\begin{proof} If $G$ has an odd cycle $\Gamma$ say, $(1, 2, \dots, 2k-1,1)$, we order the edges so that $\{1, 2\}, \{2,3\}$, $\dots$, $\{2k-1,1\}$ are the first $2k-1$ edges. Define the distance $d(j, \Gamma)$ from a vertex $j$ to the cycle $\Gamma$ to be length of the shortest path to any vertex of $\Gamma$. Let $G_s$ be the subgraph of $G$ such that $V(G_s)$ consists  of vertices of distance $\le s$ from $\Gamma$ and $E(G_s)$ consists of edges of $G$ incident only on the members of $V(G_s)$.

We shall prove that $J_{G_s}$ can be reduced to the identity for every $s$, by using induction on $s$. The assertion is true for the $s=0$ case as we have observed in Example~\ref{E:basic}. Suppose that $J_{G_s}$ has been reduced to the identity and consider $J_{G_{s+1}}$. We may order the vertices and the edges so that those of $G_s$ precede the others that do not belong to $G_s$. Then by the induction hypothesis, the primary block of $J_{G_{s+1}}$ corresponding to $J_{G_s}$ can be reduced to the identity matrix. Denote by $A$ the resulting matrix of this reduction. For every $i \in V(G_s)$, $A$ has $\be_i$ as a row.

Now, consider a vertex $j$ with distance $s+1$ from $\Gamma$. There exists $i \in V(G_s)$ that is adjacent to $j$, so that $A$ has both $\be_i$ and $\be_i+\be_j$ as rows.
 Subtracting the former from the latter gives $\be_j$.
\end{proof}

\begin{prop}\label{P:rank-spanning-tree} Suppose $G$ has only even cycles. Then the rank of $G$  is equal to the rank of a spanning tree. In particular, $\rk(J_G) = n-1$.
\end{prop}
\begin{proof} Assume by reordering the vertices so that $(1,2,\dots, 2k, 1)$ is a cycle in $G$. By reordering the edges, $J_G$ is of the form
\[
\left[
\begin{array}{cccccccccccccc}
1 & 1 & 0 & \cdots & \cdots & \cdots & 0\\
0 & 1 & 1&   0    & \cdots & \cdots & 0  \\
0 & 0  & 1&  1 &      0    & \cdots & 0 \\
\vdots& &  & \ddots   & &   &          0\\
1 & 0 &\cdots     &   0 &    1  & 0& 0 \\
* & * &  *           &  *    &    * &  * & *
  \end{array}
\right]
\]
where the last row means a submatrix of appropriate size. The penultimate row is $\row_{2k}$, and
 we have observed before that $\row_{2k} + \sum_{i=1}^{2k-1}(-1)^i \row_i = 0$. This means that one can delete the edge $\{1,2k\}$ without altering the rank of $J_G$. Hence once can cut all (even) cycles to obtain a spanning tree $T$ without altering the rank of the incidence matrix. Hence $\rk(J_G) = \rk(J_T) = n-1$.
\end{proof}

Summing up our findings so far, we conclude that:

\begin{theorem}\label{T:rank} Let $G$ be a connected $3$-colored graph. The rank of $J_G$ equals $n$ if and only if $G$ contains an odd cycle or a colored vertex. Otherwise, the rank is $n-1$.
\end{theorem}

\section{Characteristic polynomial of $\mathcal J_n$}\label{S:char}

We turn our attention back to the  hyperplane arrangement $\mathcal J_n$ from \cite{SONG}, recalled in the introduction. As in the proof of Theorem 1,  we decompose a $3$-colored graph $G$  into  three subgraphs $G', G'', G'''$:

\begin{enumerate}
\item ({\it graph of the first kind}) $G'$ is the union of colorless connected components;
\item ({\it graph of the second kind}) $G''$ is the union of isolated colored vertices;
\item ({\it graph of the third kind}) $G''' = G\setminus(G'\cup G'')$ is the union of the connected components with at least one colored vertex and at least one edge.
\end{enumerate}

Then due to Theorem~\ref{T:rank}, $G''$ and $G'''$ are of full ranks, and a component of $G'$ has full rank if and only if it has an odd cycle. To summarize:

\begin{theorem}\label{T:main2} Let $\delta$ denote the number of bipartite components of $G'$. Then the rank of $G$ equals $n - \delta$.
\end{theorem}

Recall the definition of the centrality of $3$-colored graph, which is dual to the balancedness of the signed graph in \cite{Zas12} (See \cite[Section 2]{SONG}).

\begin{defn} \label{D:central} A $3$-colored graph $([n], E)$ is said to be ${\it central}$ if
\begin{enumerate}
\item if $v$ is colored, then it is not on an odd cycle, and;
\item $\gamma(v) = \gamma(v')$ (resp. $\gamma(v) \ne \gamma(v')$) for any pair of colored vertices $v, v'$ such that there is a $v-v'$ path of even (resp. odd) length.
 \end{enumerate}
\end{defn}

\begin{corollary}
\begin{enumerate}
\item  $G'$  is always central and its rank is $|V(G')| - \delta$ where $\delta$ is the number of bipartite components of $G'$;
\item $G''$  and $G'''$ are always of full rank.
\end{enumerate}
\end{corollary}

\subsection{A formula in terms of the number of graphs} \label{S:formula}
\begin{defn}\label{D:card}
 \begin{enumerate}
\item The cardinality of a colored graph is the sum of the number of edges and the number of colored vertices;
\item $E'_{m,k,s}$ is the set of the connected, bipartite (colorless) graphs  on $[m]$ whose rank is $k$ and cardinality is $s \ge 1$;
\item $E'''_{m,k,s}$ is the set of the connected, central graphs of the third kind on $[m]$ whose rank is $k$ and cardinality is $s$.
\end{enumerate}
\end{defn}

We gather several basic properties that will be used frequently. First, note that $E'_{m,k,s}$ is empty unless $s \ge k-1$ (connectedness) and $k = m-1$ (bipartite).

\begin{lemma}  $E'_{m,k,m-1}$ is the set of trees of rank $k$. Since trees are never of full rank (Lemma~\ref{L:tree}), it is nonempty if and only if $k = m-1$. We let $\tau_m$ denote the cardinality of $E'_{m,m-1,m-1}$.
\end{lemma}

\begin{lemma} $E'''_{m,k,s}$ is empty unless $m = k$.
\end{lemma}

\begin{proof} Recall that the c-incidence matrix of a colored vertex is an $s \times m$ matrix. Since any connected graph has $\ge m-1$ edges and any graph of the third kind has at least one colored vertex, $s \ge m$. If $E'''_{m,k,s}$ is nonempty, then a member of it has full rank, so $k = m$.
\end{proof}

\begin{defn}
\begin{enumerate}
\item
Let $\nu'_{k,s}$  denote the number of connected non-colored graphs  on $[k]$, whose rank is necessarily $k$  and size is $s$;

\item
Let $\nu^b_{k,s}$  denote the number of connected bipartite non-colored graphs  on $[k]$, whose rank is necessarily $k-1$  and size is $s$;

\item
Let  $\nu'''_{k,s}$ denote the number of connected central graphs of the  third kind on $[k]$, whose rank is necessarily  $k$ and cardinality is $s$;

\item Let $\nu''_{k,s}$ denote the number of graphs of the second kind on $[k]$ whose cardinality is $s$. Note that any such graph is of rank $k$ and the cardinality also equals $k$ since there is no edge and every vertex is colored i.e. $\nu''_{k,s} = 0$ unless $k = s$.
\end{enumerate}
\end{defn}

There is nothing much to do when it comes to counting the number of central graphs of the second kind: The following lemma is self-evident.
\begin{lemma} $\nu''_{k,k} = 2^k$.
\end{lemma}

Now we can state the main result of this paper:

\begin{theorem}\label{T:formula}
The number of central graphs on $[n]$ of rank $k$ and cardinality $s$ is
\[
 \begin{array}{cll}
\gamma_{k,s} = &
\sum_{n, s}\binom n{n_b,n_{nb},n_2,n_3} 2^{n_2} \cdot \left( \sum_{n_b, s_b} \binom{n_b}{n_{b1},n_{b2},\dots,n_{b\ell}}' \prod\nu^b_{n_{bi}, s_{bi}}\right) \\
& \cdot \left( \sum_{n_{nb}, s_{nb}} \binom{n_{nb}}{n_{nb1},n_{nb2},\dots,n_{nb\ell'}}' \prod(\nu'_{n_{nbi}, s_{nbi}} - \nu^b_{n_{nbi},s_{nbi}})\right) \\
& \cdot \left(\sum_{n_3, s_3} \binom{n_3}{n_{31},n_{32},\dots,n_{3\ell_3}}'\prod \left(\sum_{t_i=1}^{s_{3i}-n_{3i}} 2 \nu^b_{n_{3i}-1, s_{3i}-t_i} \binom{n_{3i}}{t_i} \right)\right)
\end{array}
\]
where the sum runs over all partitions
\[
\begin{array}{c}
n = n_b+n_{nb}+n_2+n_3; s = s_b+ s_{nb}+ s_2+s_3 \\
n_b=\sum_{i=1}^\ell n_{bi}; s_b = \sum_{j=1}^\ell s_{bj} \\
n_{nb}=\sum_{i=1}^{\ell'} n_{nbi}; s_{nb} = \sum_{j=1}^{\ell'} s_{nbj} \\
n_3=\sum_{i=1}^{\ell_3} n_{3i}; s_3=\sum_{j=1}^{\ell_3} s_{3j} \\
\end{array}
\]
such that the rank condition $k = n - \ell$ is satisfied (see Theorem~\ref{T:rank}).
\end{theorem}

\begin{proof} Since the order, the rank and the cardinality are all additive on the connected components, we shall enumerate the number of connected components satisfying suitable conditions. That is
\begin{enumerate}
\item  Decompose $G = G' \coprod G''\coprod G''' = G'_b \coprod G'_{nb} \coprod G'' \coprod G'''$; $n = n_1 + n_2 + n_3 = n_b + n_{nb} +  n_2 + n_3$; $k = k'_b + k'_{nb}+ k_2+ k_3$; $s = s_b + s_{nb} + s_2 + s_3$, where $G'_{nb}$ (resp. $G'_b$) has non-bipartite (resp. bipartite) components.
\item Decompose $G'_b= \coprod_{i} G'_{bi}$ where $G'_{bi}$ are connected components; $n_b = \sum  n_{bi}i$; $k_b = \sum k_{bi}-1$; $s_b= \sum  s_{bi}$. Note that, necessarily, $k_{bi} = n_{bi} - 1$.
\item Similarly decompose $G'_{nb}$, $G''$ and $G'''$ into connected components. Note that, necessarily, $n_{nbi} = k_{nbi}$ and $n_{3i} = k_{3i}$ (full rank).
\end{enumerate}

\noindent (Step I)
Let's first enumerate the possible $G'_b$ on  $n_b$ given vertices (among $[n]$).   $G'_b$ can be decomposed into bipartite connected components $G'_{bi}$ on $n_{bi}$ vertices. We assume that $n_{bi} \ge n_{b \, \, i+1}$. We enumerate the possible decompositions. First we choose from $n_b$ vertices the $n_{bi}$ vertices on which connected components are built: the number of ways to do this is $\binom{n_b}{n_{b1},n_{b2},\dots,n_{b\ell}}$ (multinomial coefficient)
\[
\binom{n_b}{n_{b1},n_{b2},\dots,n_{b\ell}}= \binom{n_b}{n_{b1}}\binom{n_b-n_{b1}}{n_{b2}}\cdots \binom{n_b-n_{b1}-\cdots - n_{b \, \ell-2}}{n_{b \, \ell-1}} = \frac{n_b!}{n_{b1}!n_{b2}!\cdots n_{b \ell}!}.
\]
Since the connected components are not  ordered, if some $n_{bi}$'s are equal, say $n_{b1} = n_{b2} = \cdots = n_{bk}$, we divide the multinomial coefficient by $k!$. So we define

\begin{defn}\label{D:red-multinomial}  The {\it reduced multinomial} is
 \[
\binom{a}{a_1,a_2,\dots,a_m}' = \binom{a}{a_1,\dots,a_m}\frac1{k_1!k_2!\cdots k_r!}
\]
if
\[
a_1 = \dots = a_{k_1} > a_{k_1+1} = a_{k_1+2} = \cdots = a_{k_1+k_2} > a_{k_1+k_2+1} = \cdots.
\]
\end{defn}
Then we enumerate how many $G'_{bi}$ there are: this is precisely $\nu'_{n_{bi}, s_{bi}}$. Hence
the number of all possible $G'_b$ is
\[
\sum_{n_b=\sum_{i=1}^\ell n_{bi}}\sum_{s_b = \sum_{j=1}^\ell  s_{bj}} \binom{n_b}{n_{b1},n_{b2},\dots,n_{b\ell}}' \prod\nu'_{n_{bi}, s_{bi}}
\]
where the sums run over all possible partitions of $n_b$ and $s_b$ such that
\[
k_b = \sum (n_{bi}-1) = n_b - \ell.
\]

\noindent The non-bipartite case is similar, except that the rank is full i.e. $k_{nb} = n_{nb}$.

\noindent (Step II) When $n_2$ vertices are given, there are $2^{n_2}$ ways to color them.

\noindent (Step III) Given $n_3$ vertices, we enumerate possible $G'''$ as we did $G'$s. That is, we partition the vertices and the ranks, and enumerate the possible connected central graphs of the third kind and of the given rank and the cardinality. So, we have
\[
\sum_{n_3=\sum_{i=1}^{\ell_3} n_{3i}}\sum_{s_3=\sum_{j=1}^{\ell_3} s_{3j}} \binom{n_3}{n_{31},n_{32},\dots,n_{3\ell_3}}' \prod \nu'''_{n_{3i}, s_{3i}}
\]
where the sum runs over all partitions of $n_3$ and $s_3$. We shall prove the equality
$
\nu'''_{n_{3i},s_{3i}} = \sum_{t_i=1}^{s_{3i}-n_{3i}} 2 \nu^b_{n_{3i}-1, s_{3i}-t_i} \binom{n_{3i}}{t_i}
$
separately  in Theorem~\ref{T:main} below.

Now we put these all together. We partition $n = n_b + n_{nb} + n_2+ n_3$ and choose vertices accordingly, for which we have
\[
\binom{n}{n_b,n_{nb},n_2,n_3}
\]
choices. Once we have chosen the vertices, then we just need to enumerate the possible $G', G'', G'''$ which we have done above.
\end{proof}

Now we prove the important formula relating the number of colored graphs with the number of non-colored graphs:
\begin{theorem}  We have
\[
 |E'''_{k,k,s}| =  \sum_{t=1}^{s-k+1} 2 |E'_{k,k-1,s-t}| \binom kt.
 \]
\end{theorem}

\begin{proof} For a graph $G \in E'''_{k,k,s}$, let $_fG$ denote the graph obtained by forgetting all colors of vertices. We shall prove that $_fG$ is a member of $E'_{k,k-1,s-t}$ for some $1 \le t \le s-k+1$, and that for each $G' \in E'_{k,k-1,s-t}$, there are precisely $2 \binom kt$ graphs $G \in E'_{k,k-1,s-t}$ such that $_fG = G'$.

Let $G \in E'''_{k,k,s}$ and suppose that $G$ has  $t$ colored vertices and $s-t$ edges.
First, note that since $G$ is connected, $s-t \le k-1$ or equivalently, $t \le s-k+1$.
Since $G$ is central, connected and has a colored vertex, $G$ has no odd cycle. Thus $_fG$ is bipartite and of rank $k-1$.
Now, given $G' \in E'_{k,k-1,s-t}$, there exist precisely
\[
\binom kt \cdot 2
\]
ways to color $t$ vertices so that the resulting colored graph is in $E'''_{k,s}$: Choose a set of $t$ vertices from $G'$. For the colored graph to be central, the coloring choice of any one of the $t$ vertices determines the colors of the rest of the chosen vertices since the graph is connected.
The asserted formula follows immediately.
\end{proof}

\bibliographystyle{amsplain}
\bibliography{generate}

\end{document}